\documentclass[letterpaper]{article}
\usepackage{setspace}

\pdfoutput=1

\usepackage[english]{babel}
\usepackage[utf8x]{inputenc}
\usepackage[T1]{fontenc}

\usepackage[a4paper,margin = 1in]{geometry}

\usepackage{comment}
\usepackage{amsmath}
\usepackage{bussproofs}
\usepackage{mathtools}
\usepackage{tikz-cd}
\usepackage{amssymb}
\usepackage{amsthm}
\usepackage{amsfonts}
\usepackage{xurl}
\usepackage{bussproofs}
\usepackage{graphicx}
\usepackage[most]{tcolorbox}
\usepackage[T1]{fontenc}
\usepackage[colorinlistoftodos]{todonotes}
\usepackage[colorlinks=true, allcolors=blue]{hyperref}
\usepackage{subfig}
\usepackage{cleveref}
\usepackage{dirtytalk}
\usepackage{csquotes}
\usepackage[backend=biber, style=alphabetic]{biblatex}
\usepackage{ifpdf}
\ifpdf
\else
  
\fi
\hypersetup{
    colorlinks,
    citecolor=black,
    filecolor=black,
    linkcolor=black,
    urlcolor=black,
    linktoc=all
}

\theoremstyle{definition}
\newtheorem{theorem}{Theorem}[section]

\newtheorem{corollary}{Corollary}[theorem]

\newtheorem{proposition}[theorem]{Proposition}

\newtheorem{definition}[theorem]{Definition}

\newtheorem{example}[theorem]{Example}
 
\newtheorem{remark}{Remark}

\title{Tarski's least fixed point theorem: \\ A predicative type theoretic formulation}
\author{Ian Ray \\
ibray1115@gmail.com}

\date{December 2023}

\DeclareMathOperator{\dB}{\downarrow^B}
\DeclareMathOperator{\dBV}{\downarrow^B_\mathcal{V}}
\DeclareMathOperator{\dA}{\downarrow^A}

\begin{document}

\maketitle

\begin{abstract}
    We provide a type theoretic treatment of the paper \say{On Tarski's fixed point theorem} by Giovanni Curi. There are benefits to having a type theoretic formulation apart from routine implementation in a proof assistant. By taking advantage of (higher) inductive types, we can avoid complicated set theoretic constructions. Arguably, this results in a presentation that is conceptually clearer. Additionally, due the predicative admissibility of (higher) inductive types we take a step towards the \say{system independent} derivation that Curi calls for in his conclusion. Finally, we explore a condition on monotone maps that guarantees they are `generated' and give an alternative statement of the least fixed point theorem in terms of this condition. 
\end{abstract}

\section*{Acknowledgments}
First and foremost I would like to thank Giovanni Curi. His correspondence was essential for getting my bearings in this area of mathematics. I would also like thank Tom de Jong. He has already established much of predicative and constructive order theory (as well as domain theory) in a univalent setting. Beyond this he has provided me with countless hours of discussion and insight. Without Tom this paper would not exist. Finally, I would like to thank Mart\'in Escard\'o, Ulrik Buchholtz, Jon Sterling, Jem Lord and many others from the Univalent Agda discord for small and large contributions to the present work.

\tableofcontents

\pagebreak

\section{Introduction}
\label{sec:1}
Fixed point theorems, although of great interest and significance, are notorious for their non-constructive nature (e.g. Brouwer's Fixed Point Theorem \cite{Shul2}). Certain fixed point theorems do admit intuitionistic proofs (see \cite{Coquand}). Of particular interest is the work \cite{Esc} which contains a constructive proof a Pataraia's fixed point theorem: an analog of Tarski's fixed point theorem for directed complete posets (DCPOs). However, some still view proofs like the ones above as not entirely constructive, due to their impredicative nature. What exactly is impredicativity? To assist explanation we focus our attention to one part of Tarski's fixed point theorem: every monotone map $f$ on a complete lattice $L$ has a least fixed point. Traditional approaches define the least fixed point to be the infimum over the set of `deflationary' points. Specifically, the least fixed point is
$$\bigwedge \{ x \in L \ | \ f(x) \leq x \}.$$
Note that the element we are trying to define is itself deflationary (as it is a fixed point) and thus an element of the very set we take the infimum over. This is an example of an impredicative construction. Alternatively, a predicative construction can be loosely characterized as one which builds an object up from things below. We will continue a critical discussion of predicativity in \Cref{sec:2}. 

Giovanni Curi provides a predicative proof of a variation to Tarski's least fixed point theorem \cite{Curi} using constructive set theory (CZF). In Section 2 of his paper Curi recounts a paradigm shifting result in predicative order theory: any complete lattice with small carrier is necessarily trivial (for a type theoretic proof, see \cite{Jong2021OnST}). For this reason we are forced to work with large carriers if we want to study non-trivial complete lattices in a predicative setting. It is worth explicating that in the framework of CZF the small/large dichotomy is encoded by sets and proper classes. Curi then restricts attention to complete lattices that have a set of generators (along with other strong assumptions) to salvage a least fixed point theorem. The imposed restrictions will allow one to build the least fixed point inductively. Of course, in the presence of impredicative axioms the additional assumptions are satisfied vacuously. Curi concludes by observing his works dependence on the set theoretic framework, even stating that a \say{system independent} derivation is desirable. 

Currently, we do not have the facilities to provide a derivation that is completely independent of any particular system. One may speculate that the notion of a predicative topos will provide canonical models of predicative mathematics --- much in the same way that elementary topoi provide models of intuitionisitic mathematics. Progress has been made in this direction, but there is no agreed upon notion yet in place (see \cite{denBerg}). Alternatively, a translation of \cite{Curi} into type theory may be a step towards such a system independent derivation; because type theory is a structural foundation. 

There are additional benefits of working in a type theoretic setting. First, inductive constructions are first class citizens; so we don't have to manually construct the object of interest as in \cite{Curi}. This allows us to cut away the complications that sideline the main arguments in \cite{Curi}. For this reason, a type theoretic derivation is arguably cleaner and more intuitive than the set theoretic counter part. Second, the proof can be routinely verified with a proof assistant. In fact, every stated definition and proposition in the present work has been implemented in Agda \cite{Esc2, Ian}. An off shoot of this second benefit is that one could, in theory, port the proof into Cubical Agda and compute fixed points. Third, a proof in MLTT* + HITs is valid in any Grothendieck $\infty$\nobreakdash-topos \cite{Shul3} (MLTT* = MLTT + FunExt + PropExt, but not including full univalence \cite{Uni,Rijke}). Finally, it is worth noting that the type theoretic proof is universe polymorphic (see \Cref{sec:3}) and thus, technically, more general than a proof in a set theoretic framework. We intend to work informally with type theory much like how mathematicians work informally with sets. The ideal audience of this paper is a (univalent) type theorist, but the informal style paired with no appeals to higher types will lend accessibility to any interested party.

In \Cref{sec:2} we commence a more careful discussion of predicativity. In \Cref{sec:3} we fix the type theoretic framework and some conventions we employ to avoid technicalities. If you are unfamiliar with type theory then you may reasonably skip \Cref{sec:3} and still understand the bulk of the paper. \Cref{sec:4} develops a bit of order theory. In particular, we define posets, sup lattices and the notion of a basis. \Cref{sec:5} describes, what Curi calls, abstract inductive definitions in the language of type theory. Abstract inductive definitions serve to carve out a least subset of the basis that has some nice closure properties. As we will explore in \Cref{sec:7} there is a correspondence between deflationary points and a certain class of subsets that possess similar closure properties. The least closed subset will be closely related to the least fixed point via this correspondence. In Curi's work this least closed subset is constructed manually via set theoretic axioms, while in type theory it manifests itself as a special quotient inductive type (QIT). In \Cref{sec:6,sec:8,sec:9} we impose further restrictions on abstract inductive definitions (viz.\ local and bounded) and sup lattices (viz.\ small presented). These restrictions serve to salvage a type theoretic formulation of the least fixed point theorem which is stated in \Cref{sec:10}. We conclude the paper by investigating a condition on monotone maps that guarantees they have a least fixed point.

\section{Note about Predicativity}
\label{sec:2}
Unfortunately, the history behind predicativism is murky and wrought with philosophical motivations. In fact, there are many conflicting notions of predicativity; each with their own supporters and justifications. Classical notions of predicativity where explored by Poincar\'e, Russell and Weyl \cite{Cros3, Cros}. More recently Kreisel, Feferman and Sch$\ddot{\text{u}}$tte provided a proof theoretic analysis of predicativity \cite{Cros2, Cros}. When we say `predicative' here we mean what is often referred to as constructive or generalised predicativity \cite{Cros, Cros3}. A generalised predicative system goes beyond Fefferman's predicativity. In particular, generalised predicative systems have infinitary inductive types. 

The systems of Constructive Zermelo Fraenkel set theory (CZF) and Martin L$\ddot{\text{o}}$f type theory (MLTT) are examples of such predicative systems. The Powerset axiom is an example of an impredicative principle and as such is not included in CZF. Similarly, the principle of propositional resizing is viewed as impredicative and thus not included in MLTT or any extensions we consider. In the absence of these strong principles we instead rely on inductive constructions. These constructions are argued to be philosophically sound with respect to such predicative systems \cite{Cros}. When working in extensions of MLTT we may apply similar philosophical arguments to justify the predicative admissibility of higher inductive types (HITs) as well. One may wish to avoid philosophical justifications entirely. One such justification would be a reduction of type theory extended by HITs to a system that is accepted as predicative. Such reductions have yet to be completely spelled out, but do seem possible. For example, if \cite{Jon} could be extended to the schemas in \cite{Coq2} in a constructive and predicative meta theory then we may achieve such a reduction (this observation is due to Ulrik Buchholtz). 

Now that we have clarified the meaning of predicativity we want to return to the discussion of the standard formulation Tarski's least fixed point theorem: every monotone map $f$ on a complete lattice $L$ has a least fixed point. The standard construction (see \Cref{sec:1}) does not work in systems like CZF or MLTT*. Consider a non-trivial complete lattice $L$ and a monotone map $f : L \to L$, and observe that, since the carrier $L$ is necessarily large, the collection $\{x \in L \ | \ f(x) \leq x\}$ is not provably small and as such the join does not exist. In \cite{Curi} a predicative formulation of Tarski's least fixed point theorem is proved. Although, as we shall see, this formulation contains serious concessions. It is very important to clarify: we have only argued that the standard construction is not predicatively viable, but it may be that some other construction is. It seems unlikely that there is a predicative proof of the standard formulation of Tarski's least fixed point theorem, but ideally we would show that such a proof is impossible. We are actively exploring the prospect of showing that the standard formulation of Tarski's least fixed point theorem is too strong of a result to have in a predicative theory --- a predicative taboo so to speak. One could do this by showing that the standard formulation implies some impredicative axiom or by providing a model where the standard formulation fails. Showing that the standard formulation of Tarski's least fixed point theorem is a predicative taboo would fully justify the content of this paper and that of \cite{Curi}. Regardless, we still feel that the failure of the standard construction is sufficient motivation to investigate a predicative version of Tarski's least fixed point theorem.

\section{Type Theoretic Framework}
\label{sec:3}
One of the main goals of this paper is to maintain accessibility. For this reason, we will work informally with type theory much like a traditional mathematician works with set theory. Thus, in many occasions the technical type theoretic details of a proof will be underspecified. The purpose of this section is to assist in filling in these technical details to the interested parties and as such can be reasonably skipped by non-type theorists. 

In traditional treatments of type theory one would find closure properties of universes as well as type formation, introduction and elimination rules stated in the language of natural deduction. We assume familiarity with such formal treatments (for more see \cite{Rijke, Uni}). Our starting point is Martin Löf Type Theory (MLTT): which has an empty type ($\boldsymbol{0}$), a unit type ($\boldsymbol{1}$), the natural numbers ($\mathbb{N}$), binary coproducts ($+$), dependent pair types ($\sum$), dependent functions types ($\prod$), intensional identity types ($=$) and general inductive types ($\text{W}$). We assume the existence of a base universe $\mathcal{U}_0$ as well as the operations successor $(\_)^+$ and join $\_ \sqcup \_$ which satisfy the expected definitional equalities (see \cite[Section 2.1]{Jong2023DomainTI}). We attain an infinite tower of universes $\mathcal{U}_0 , \mathcal{U}_1 , \dots$ but we will often work with arbitrary universes such as $\mathcal{U, V, W, T}, \dots$. We do not assume cumulativity of universes (e.g. for any $A : \mathcal{U}_i$ we have $A : \mathcal{U}_{i+1}$), but instead we employ a lifting operation: $\text{Lift}_{\mathcal{U,V}} : \mathcal{U} \to \mathcal{U \sqcup V}$ (see \cite[Section 2.1]{Jong2023DomainTI}). Universes will be presented à la Russell for convenience but should be understood to have an encoding à la Tarski. In the current state of development univalence is not required, but we do assume propositional and function extensionality as well as propositional truncation. There are some nuances in our base type theory that we should explicate. To maintain predicativity, we do not assume any resizing principles except explicitly in \Cref{cor:Impredicative-Least-Fixed-Point}. The propositional resizing principle, $\text{Prop-Resizing}_{\mathcal{U}, \mathcal{V}}$, states that any proposition in the universe $\mathcal{U}$ is equivalent to a proposition in the universe $\mathcal{V}$ (see \cite{Jong2021OnST}). Additionally, we assume the existence of a special sort of quotient inductive type (QIT). Such an assumption may at first seem excessive, but investigation of this QIT reveals it is extremely tame comparatively. We will revisit this discussion in Section 4. The notion of equivalence serves as the standard notion of sameness in type theory (for a precise definition see Section 9 of \cite{Rijke}). We say a function $f : A \to B$ is an equivalence when it has a left and right inverse and denote the type of such equivalences as $A \simeq B$. Finally, given universes $\mathcal{U}$ and $\mathcal{V}$, we say that a type $Y : \mathcal{U}$ is $\mathcal{V}$-small if there is a type $X : \mathcal{V}$ such that $X \simeq Y$ and a type $Y : \mathcal{U}$ is locally $\mathcal{V}$-small if $x = y$ is $\mathcal{V}$-small for all $x, y : Y$. In our type theoretic formulation the small/large dichotomy is encoded via those types that are $\mathcal{V}$-small and those that are not (in \Cref{sec:3} we fix some universes and the significance of the universe $\mathcal{V}$ will become clear).

We now discuss some conventions that will allow us to avoid technical type theoretic proofs. For example, the notion of transport is often motivated by the principle of indiscernibility of identicals. This principle, often taken for granted by mathematicians, says identified elements can be substituted for each other in expressions. Another type theoretic concept, often taken for granted, is the application of $f : X \to Y$ to an identification $x = x'$ which yields an identification $f(x) = f(x')$. As we are doing set-level mathematics we will employ such notions tacitly. Propositional truncation is a somewhat technical part of any formal treatment of type theory (see Section 14 of \cite{Rijke}). In particular, given a type $A$ we have a proposition $||A||$ called the propositional truncation of $A$. Given a proposition $Q$ we can define a map $|| A || \to Q$ by giving a map $A \to Q$. Propositional truncations can be formulated via universal properties or recursion principles. The observant reader will notice that when we must define a map out of a truncated type, the codomain is almost always a proposition. For that reason we will not explicitly mention that we are applying propositional truncation recursion and we will leave notions that are defined in terms of truncation (like surjections) underspecified. In the same vein we do not specify when we are applying propositional or function extensionality.

We will frequently work with subsets in this paper and as such we need to address how formal we wish to be. In a standard treatment, ignoring universes, a subset of a type $X$ is a term $S$ in the powerset of $X$, $\mathcal{P}(X) :\equiv X \to \text{Prop}$. We write $x \in S$ for the underlying type of $S(x)$. Given two subsets $S, S' : \mathcal{P}(X)$ we write $S \subseteq S'$ for the type of dependent functions $\prod_{x : X} x \in S \to x \in S'$. We wish for mathematicians who are unfamiliar with the type theoretic encoding to follow our arguments. For this reason we will avoid using lambda notation when talking about subsets. For example, we define the union, $\bigcup \alpha$, of a family of subsets $\alpha : I \to \mathcal{P}(X)$, as $x \mapsto \exists i : I, x \in \alpha(i)$.
When $X$ is a set we can define the singleton $\{a\}$, for some $a : X$, as $x \mapsto a = x$. Finally, since we define suprema (see \Cref{sec:3}) in terms of families rather than subsets we also need to fix some notation that will unify the two concepts. Given a subset $S : \mathcal{P}(X)$ we define the total space of $S$, $\mathbb{T}(S) : \equiv \Sigma_{x : X} x \in S$, with inclusion into $X$ given by $\text{inc}_X :\equiv \text{pr}_1 : \mathbb{T}(S) \to X$. It is worth recalling that using sigma notation in this way is analogous to set builder notation (e.g.\ $\{ x \in X \ | \ x \in S \}$). This will inform many of our translations of Curi's work into type theory. We will often conflate a subset and its total space to aid in readability, but one should bare in mind the technical difference between them. Finally, since we are working predicatively we will not ignore universe levels. For every universe $\mathcal{U}$, there is a type of propositions $\Omega_\mathcal{U} :\equiv \sum_{A : \mathcal{U}} \text{is-prop} (A)$. For any type $X$ we may consider the $\mathcal{U}$-powerset of $X$, $\mathcal{P}_\mathcal{U}(X) :\equiv X \to \Omega_\mathcal{U}$. We call the elements of $\mathcal{P}_\mathcal{U}(X)$ the $\mathcal{U}$-valued subsets of $X$. The above discussion extends to $\mathcal{U}$-valued subsets as well.

\section{Sup Lattices and Small Bases}
\label{sec:4}
We commence this section by briefly recounting the notion of a set generated complete lattice as formulated in \cite{Curi}. A complete lattice $L$ is set-generated if it has a subset $B \subseteq L$ such that, for any $x \in L$,
\begin{enumerate}
    \item $\downarrow^B x = \{ b \in B \ | \ b \leq x \} $ is a set.
    \item $x = \bigvee \downarrow^B x$.
\end{enumerate}

We now work towards translating this notion into type theory. As we have established, it is well known that in a predicative setting there are no non-trivial examples of small (sufficiently) complete posets (see \cite{Jong2021OnST}). Thus, when working predicatively it is necessary to define sup lattices, and by extension, posets, in a universe polymorphic manner. 

\begin{definition}
    \label{def:poset}
    A \textbf{partially ordered set} (Poset) consists of a type $P : \mathcal{U}$ and a propositional valued relation $\_ \leq \_ : P \to P \to \mathcal{W}$ which satisfies:
    \begin{enumerate}
        \item $x \leq x$ (reflexivity)
        \item $x \leq y \to y \leq x \to x = y$ (anti-symmetry)
        \item $x \leq y \to y \leq z \to x \leq z$ (transitivity)
    \end{enumerate}
    for any $x, y, z : P$. 
\end{definition}    

The anti-symmetry assumption guarantees that $P$ is a set (see Lemma 3.2.3 of \cite{Jong2023DomainTI} for a proof following from Hedberg's Lemma). Given universes $\mathcal{U}$ and $\mathcal{W}$ we denote the type of Posets with carrier in $\mathcal{U}$ and order valued in $\mathcal{W}$ as $\text{Poset}_{\mathcal{U}, \mathcal{W}}$.

\begin{definition}
    \label{def:endomap}
    Given a poset $P : \mathcal{U}$ an endomap $f : P \to P$ is \textbf{monotone} if for any $x , y : P$ with $x \leq y$ we have $f(x) \leq f(y)$.
\end{definition}  

Of course it is easy enough to extend this to functions between different posets, but we focus our attention to endomaps as our goal is to study fixed points. We now define the notion of sup lattice. We follow the reasonable convention of \cite{Jong2023DomainTI} in phrasing completeness with respect to small families rather than subsets.

\begin{definition}
    \label{def:sup-lattice}
    A \textbf{sup lattice} is a poset $L : \mathcal{U}$ and $\_ \leq \_ : L \to L \to \mathcal{W}$ that has joins for all small families. That is, families $\alpha : I \to L$ with $I : \mathcal{V}$. We denote the join as $\bigvee \alpha$ but, if $\alpha$ is understood we will denote the join simply as $\bigvee I$.
\end{definition}   

Of course, we extend the notion of monotone endomaps to sup lattices. Given universes $\mathcal{U}$, $\mathcal{W}$ and~ $\mathcal{V}$ we denote the type of sup lattices with carrier in $\mathcal{U}$, order valued in $\mathcal{W}$ and joins of families indexed in $\mathcal{V}$ as $\text{Sup-Lattice}_{\mathcal{U}, \mathcal{W} , \mathcal{V}}$, but often times we will simply say such a lattice is a $\mathcal{V}\text{-sup-lattice}$ and omit the other universes. As stated in \Cref{sec:3}, we use the universe $\mathcal{V}$ as a point of reference for our relative notion of smallness and our language will often reflect this. In many concrete cases $\mathcal{U} \equiv \mathcal{V}^+$ and $\mathcal{W} \equiv \mathcal{V}$ (see \Cref{ex:pow-basis}). Notice that if $\mathcal{W} \equiv \mathcal{V}$ then anti-symmetry implies that $L$ is locally $\mathcal{V}$-small.

We now define, analogous to the notion of a generating set from \cite{Curi}, what it means for a sup lattice to have a basis. First recall, as stated in \Cref{sec:3}, given a type $X : \mathcal{U}$ and a subset $U : \mathcal{P}_\mathcal{T} (X)$ the total space of $U$ is $\mathbb{T}(U) :\equiv \sum_{x : X} x \in U$. 

\begin{definition}
    \label{def:basis}
    A $\mathcal{V}\text{-sup-lattice}$ $L$ has a \textbf{basis} provided there is a type $B : \mathcal{V}$ and a map $\beta : B \to L$ such that, for all $x : L$, the following is satisfied:
    \begin{enumerate}
        \item $\beta(b) \leq x$ is $\mathcal{V}\text{-small}$ for all $b : B$ and $x : L$.
        \item The family $\beta \circ \text{pr}_1 : (\sum_{b : B} \beta(b) \leq x) \to L$ has $x$ as its supremum.
    \end{enumerate}
    We now establish some helpful notation: first define $b \leq^B x$ to be the small type equivalent to $\beta(b) \leq x$ in condition 1. Then define $\dB x : \equiv \sum_{b : B} (\beta(b) \leq x)$ and $\dBV x : \equiv \sum_{b : B} (b \leq^{B} x)$ as well as inclusions into $B$ which are given by the first projection $\text{pr}_1$. Finally, we have an equivalence $\dBV x \simeq \dB x$ which is easily constructed from the equivalence between second components. With this notation established we can write condition 2 more suggestively as 
    \begin{enumerate}
        \item[2'.] $x = \bigvee \dB x$
    \end{enumerate}
    Finally, if a $\mathcal{V}\text{-sup-lattice}$ $L$ has a basis then we say it is $\mathcal{V}$\textbf{-generated}.
\end{definition}

We will typically conflate the two orders $\_ \leq \_$ and $\_ \leq^{B} \_$ and corresponding total spaces, as well as the joins of either. This is justified by \Cref{cor:equiv-same-sup} -- which essentially says that equivalences preserve joins. Notice, we have a subset $\_ \leq^B a : \mathcal{P}_\mathcal{V}(B)$ such that $\dB a \equiv \mathbb{T}(\_ \leq^B a)$. We will often conflate the total space $\dB a$ and this subset, and similarly for $\dBV a$.

\begin{example}
    \label{ex:pow-basis}
    Given a set $A : \mathcal{V}$, the powerset $\mathcal{P}_\mathcal{V} (A)$ is an example of a $\mathcal{V}$-generated sup lattice ordered by subset inclusion with joins given by arbitrary unions. Notice that $A$ together with the map $\{\_\} : A \to \mathcal{P}_\mathcal{V}(A)$ is a basis, where $\{a\}$ is defined only because $A$ is a set. For condition 1, note that for any $a : A$ and $X : \mathcal{P}_\mathcal{V}(A)$ the type $\{a\} \subseteq X$ is $\mathcal{V}$-small. For condition 2, notice
    $$\bigcup \dA X = X$$
    since $x \in \bigcup \dA X \equiv \exists (a,\_) : \downarrow^A X, x \in \{a\}$ is equivalent to $\exists a : A, \{a\} \subseteq X \land x \in \{a\}$ which is equivalent to $x \in X$. For details on the type theoretic encoding of unions see \Cref{sec:3}.

    There is an alternative basis for $\mathcal{P}_\mathcal{V} (A)$. For any type $X$ there is a type, $\text{List}(X)$, of finite lists of elements of $X$. $\text{List}(X)$ is inductively generated by the empty list and concatenation (see \cite[Chapter 6]{Uni}. Consider the map $\beta : \text{List}(A) \to \mathcal{P}_\mathcal{V} (A)$ defined as follows
    \begin{align*}
        [ \ ] &\mapsto \emptyset \\
        x :: l &\mapsto \{x\} \cup \beta(l).
    \end{align*}
    It is routine to show that $\beta : \text{List}(A) \to \mathcal{P}_\mathcal{V} (A)$ is also a basis for $\mathcal{P}_\mathcal{V} (A)$.
\end{example}

We will record the fact that taking the supremum preserves containment of subsets. 

\begin{proposition}
    \label{prop:subset-preserve-sup}
    Consider a $\mathcal{V}$-sup-lattice L a type $A : \mathcal{V}$ and a map $m : A \to L$. Further suppose we have $S, R : \mathcal{P}_\mathcal{V}(A)$ with $S \subseteq R$ then $\bigvee S \leq \bigvee R$. Here $m \circ \text{inc}_A : \mathbb{T}(S) \to L$ is the map we are leaving implicit and similarly for $\mathbb{T}(R)$. 
\end{proposition}

\begin{proof}
    Clearly, both suprema exist so it suffices to show that $\bigvee R$ is an upper bound of $S$. To this end let $b : B$ with $b \in S$. Of course, since $S \subseteq R$ we may conclude $b \in R$ and thus, $m(b) \leq \bigvee R$. 
\end{proof}

In line with earlier observations; it is well-known that large lattices containing all small suprema do not necessarily contain all small infima. One motivation for working with a basis is that the resulting sup lattices are better behaved. In particular we have the following result. 

\begin{proposition}
    A $\mathcal{V}$-generated sup lattice $L$ with a basis $\beta : B \to L$ has all infima. 
\end{proposition}
            
\begin{proof}
    Consider a family $\alpha : I \to L$. A classical proof would proceed by defining the inf of $I$ to be the sup of all the lower bounds of $I$. This argument can be replicated if we take advantage of our basis. Consider the subset $X : \mathcal{P}_\mathcal{V}(B)$ defined as $x \mapsto \prod_{i : I} x \leq^B \alpha(i))$ (which can be mapped to the carrier via $\beta \circ \text{inc}_B : \mathbb{T} (X) \to L$). We claim that $\bigvee X$ is the infimum of $I$. 
    
    We first show that $\bigvee X$ is a lower bound of $I$. For $i : I$, we must show $\bigvee X \leq \alpha(i)$. By construction of $X$, $\alpha(i)$ is an upper bound of $X$, so the desired results follows from the least upper bound condition of $X$. 
    
    We now show that $\bigvee X$ satisfies the greatest lower bound condition. Let $l$ be any other lower bound of $I$ and notice, by transitivity, $b \leq^B l \to \prod_{i : I} (b \leq^B \alpha(i))$ for any $b : B$. So we have $\dB l \subseteq X$ and
    $$ l = \bigvee \dBV l \leq \bigvee X, $$
    where equality follows from Condition 2 of \Cref{def:basis} and the inequality from \Cref{prop:subset-preserve-sup}.
\end{proof}

The next theorem and corollary are essential for our development. It says that reindexing families along a surjection (or an equivalence) does not change the supremum.

\begin{proposition}
    \label{prop:surj-same-sup}
    Consider a $\mathcal{V}$-sup-lattice $L$ and types $X : \mathcal{V}$ and $Y : \mathcal{T}$. Further suppose there is a map $m : Y \to L$ and a surjection $s : X \twoheadrightarrow Y$ then $\bigvee m \circ s$ (which exists by assumption) is also the supremum of the family $m : Y \to L$ (which does not exist in general). Since the maps are understood we could more clearly say $\bigvee X = \bigvee Y$.
\end{proposition}

\begin{proof}
    To show that $\bigvee X$ is also the supremum of the family $m : Y \to L$ it suffices to show it is an upper-bound of the family $m$ and that it satisfies the least upper-bound condition for the family $m$. For the former, let $y : Y$ with the intent to show $m(y) \leq \bigvee X$. Since $s$ is a surjection we have an $x : X$ such that $s(x) = y$. It now follows that $m(y) = m(s(x)) \leq \bigvee X$ since by assumption $\bigvee X$ is an upper-bound of $m \circ s$. For the latter, let $u : L$ be any other upper-bound of the family $m : Y \to L$. To show $\bigvee X \leq u$ it suffices to show $u$ is an upper-bound of $m \circ s$, but this follows immediately.
\end{proof}

\begin{corollary}
    \label{cor:equiv-same-sup}
    Consider a $\mathcal{V}$-sup-lattice L and types $X : \mathcal{V}$ and $Y : \mathcal{T}$. Further suppose there is a map $m : Y \to L$ and an equivalence $e : X \simeq Y$ then $\bigvee X = \bigvee Y$.
\end{corollary}

\begin{proof}
    Observe that any equivalence is a surjection so the result follows by \Cref{prop:surj-same-sup}.
\end{proof}

\section{Inductive Generators}
\label{sec:5}
The notion of an abstract inductive definition considered in \cite{Curi} generalizes the inductive definitions in \cite{Aczel}. Given a complete lattice $L$ with generating subset $B$ consider a subset $\Phi \subseteq B \times L$. In some sense $\Phi$ can be thought of as inductively defining a subset of $B$. To make this precise in set theoretic language we have to define appropriate closure conditions and then proceed to prove that there is in fact a subset satisfying these conditions. Curi does exactly that in \cite{Curi}. We briefly recount the desired closure conditions but for obvious reasons do not attempt to construct the desired subset. Consider a subclass $Y \subseteq B$,
\begin{enumerate}
    \item $Y$ is $\text{c}_L$-closed if for every subset $U \subseteq Y$ we have that $\dB \bigvee U \subseteq Y$.
    \item $Y$ is $\Phi$-closed if for every $(b , a) \in \Phi$ then $\dB a \subseteq Y \implies b \in Y$.
\end{enumerate} 
We denote the least closed class, under the $\text{c}_L$ and $\Phi$ closure conditions, as $\mathcal{I}(\Phi)$. In \cite{Curi}, Curi proves that $\mathcal{I}(\Phi)$ always exists in CZF.

We will now translate this notion into type theory. Given a $\mathcal{V}\text{-sup-lattice}$ with a $\mathcal{V}\text{-basis}$ we can encode the notion of abstract inductive definitions à la CZF. Admittedly, the terminology is a bit unfortunate here as inductive definitions/constructions are first class in type theory and thus take on a more general meaning. To ameliorate this we will call them inductive generators or simply generators. We will also depart from Curi's notation slightly, by using $\phi$ and $\mathcal{I}_\phi$, in the interest of differentiating between set theoretic and type theoretic constructions.

\begin{definition}
    \label{def:BSG}
    Given $\mathcal{V}\text{-Sup-Lattice}$ $L : \mathcal{U}$ with a $\mathcal{V}\text{-basis}$ $\beta : B \to L$ we define an \textbf{inductive generator} to be a subset $\phi : \mathcal{P}_\mathcal{U \sqcup V^+} (B \times L)$.
\end{definition}

This section contains some unavoidable use of type theoretic concepts (see \cite[Chapter 6]{Uni}). Using a generator $\phi$ as a parameter we construct a special higher inductive type (HIT) family. HITs are an active area of study so we will quickly comment on exactly what we must assume (for consistency of HITs see \cite{Shul}). We do not need the full notion of HITs for our purposes. The type we are postulating is a quotient inductive type (QIT) family, where everything is quotiented. This amounts to having a propositional truncation constructor. Of course, details about how to construct QITs in general are beyond the scope of this note. For now we will simply state the definitions of our proposed type and hope that its tameness is apparent. First we provide some shorthand for the two closure properties:

\begin{definition}
    \label{def:closed-notation}
    Given a generator $\phi : \mathcal{P}_\mathcal{U \sqcup V^+} (B \times L)$ and a subset $S : \mathcal{P}_\mathcal{T}(B)$ we say that $S$ is \textbf{closed under containment} if there is a function 
    $$\text{is-c-closed}(S) : \equiv \prod_{U : \mathcal{P}_{\mathcal{V}} (B)} U \subseteq S \to \dBV \bigvee U \subseteq S$$
    and $S$ is \textbf{closed under} $\phi$ if there is a function
    $$\text{is-}\phi\text{-closed} (S) : \equiv \prod_{a : L} \prod_{b : B} (b , a) \in \phi \to \dB a \subseteq S \to b \in S.$$
\end{definition}

What follows is our first use of Higher Inductive Types. 

\begin{definition}
    \label{def:QIT}
    Given a generator $\phi : \mathcal{P}_\mathcal{U \sqcup V^+} (B \times L)$ we define the QIT family $\mathcal{I}_\phi : B \to \mathcal{U} \sqcup \mathcal{V}^+$, which we call the \textbf{least closed subset under containment and} $\phi$, which has the following constructors:
    \begin{enumerate}
        \item $\mathcal{I}\text{-trunc} : \prod_{b : B} \text{is-prop}(\mathcal{I}_\phi(b))$,
        \item $\text{c-cl} : \text{is-c-closed} (\mathcal{I}_\phi)$,
        \item $\phi\text{-cl} : \text{is-}\phi\text{-closed} (\mathcal{I}_\phi)$.
    \end{enumerate}
    Notice that the universe level $\mathcal{U \sqcup V^+}$ of $\mathcal{I}_\phi$ results from quantification over $L : \mathcal{U}$ and $\mathcal{P}_\mathcal{V}(B) : \mathcal{V^+}$. We choose not to state the induction principle of $\mathcal{I}_\phi$ as it will not be needed beyond this section. For details see \cite{Ian}.
\end{definition}

\begin{remark}
    The $\mathcal{I}\text{-trunc}$ constructor guarantees that $\mathcal{I}_\phi : \mathcal{P}_\mathcal{U \sqcup V^+}(B)$. 
\end{remark}

The truncation constructor restricts the possible codomains to be propositionally valued families indexed by $B$ and possibly $\mathcal{I}_\phi (b)$. When defining functions out of $\mathcal{I}_\phi$ we may use the induction principle or pattern match on the point constructors c-cl and $\phi$-cl. We need to show one crucial property of $\mathcal{I}_\phi$ before moving on: $\mathcal{I}_\phi$ is initial among all subsets of $B$ closed under containment and $\phi$.  

\begin{proposition}
The subset $\mathcal{I}_\phi$ is initial with respect to small subsets, justifying its name. That is, we have a term
$$\mathcal{I}\text{-initial} : \prod_{P : \mathcal{P}_\mathcal{T}(B)} \text{is-c-closed}(P) \to \text{is-}\phi\text{-closed} (P) \to \mathcal{I}_\phi \subseteq P.$$
\end{proposition}

\begin{proof}
    By definition of $\mathcal{I}_\phi$
\end{proof}

\section{Local Inductive Generators}
\label{sec:6}
In \cite{Curi} it is desirable to indicate when an abstract inductive definition yields a set, at least locally. We say an abstract inductive inductive definition $\Phi$ is local if for every $a \in L$ the class 
$$\{ b \in B \ | \ \exists a' \in L, (b , a') \in \Phi \land a' \leq a \}$$
is a set. When an abstract inductive definition $\Phi$ is local we may define a monotone operator as follows $\Gamma_\Phi(a) \equiv \bigvee \{ b \in B \ | \ \exists a' \in L, (b , a') \in \Phi \land a' \leq a \}$.

We now translate these notions into type theory. For the remainder of this note we work in the context of a $\mathcal{V}$-sup-lattice $L$ with a basis $\beta : B \to L$.

\begin{definition}
    \label{def:S}
    Given a generator $\phi : \mathcal{P}_\mathcal{U \sqcup V^+} (B \times L)$ and $a : L$ we define the following subset of $B$ 
    $$ b \mapsto \exists a' : L , (b , a') \in \phi \land a' \leq a. $$ 
    We denote the total space of this subset as $S_{\phi,a}$.
\end{definition}

As is convention we will conflate the total space $S_{\phi,a}$ and its corresponding subset.

\begin{proposition}
    \label{prop:S-mono}
    Given a generator $\phi : \mathcal{P}_\mathcal{U \sqcup V^+} (B \times L)$ and $x , y : L$ such that $x \leq y$ we have $S_{\phi,x} \subseteq S_{\phi,y}$.
\end{proposition}

\begin{proof}
    It suffices to show that
    $$\exists a' : L , (b , a') \in \phi \land a' \leq x \to \exists a' : L , (b , a') \in \phi \land a' \leq y$$
    for each $b : B.$
    Assume we have $a' : L$ such that $(b,a') \in \phi$ and $a' \leq x$. By assumption, $x \leq y$ so by transitivity, $a' \leq y$.
\end{proof}

We can now state what it means for a generator to be local.

\begin{definition}
    \label{def:local}
    We say a generator $\phi : \mathcal{P}_\mathcal{U \sqcup V^+} (B \times L)$ is \textbf{local} if for any $a : L$ the type $S_{\phi , a}$ is $\mathcal{V}$-small.
\end{definition}

We now define, for any local generator, the following map. Here we are taking advantage of the existence of small joins.

\begin{definition}
    \label{def:Gamma}
    Given a local generator $\phi : \mathcal{P}_\mathcal{U \sqcup V^+} (B \times L)$ we define the \textbf{monotone operator} under~ $\phi$ to be the map $\Gamma_\phi : L \to L$ defined as 
    $$\Gamma_\phi (a) : \equiv \bigvee S_{\phi , a}$$
    where there is an obvious map $\beta \circ \text{inc}_B : S_{\phi , a} \to L$ (note the implicit application of \Cref{prop:surj-same-sup}).
\end{definition}

Not surprisingly this map is monotonic. It is worth noting that the formalized proof must account for the implicit use of \Cref{prop:surj-same-sup} in the Definition of $\Gamma_\phi$ which adds a layer of difficulty. This is one of the reasons why \Cref{prop:surj-same-sup} is so essential.

\begin{proposition}
    \label{prop:Gamma-mono}
    Given a local generator $\phi : \mathcal{P}_\mathcal{U \sqcup V^+} (B \times L)$ the map $\Gamma_\phi$ is monotone.
\end{proposition}

\begin{proof}
    Given $x \leq y$ we must show $\bigvee S_{\phi,x} \leq \bigvee S_{\phi,y}$. By \Cref{prop:S-mono} we have $S_{\phi,x} \subseteq S_{\phi,y}$. Thus, by \Cref{prop:subset-preserve-sup}, the desired inequality holds.
\end{proof}

Perhaps more surprisingly, every monotone map provides a canonical local generator such that the induced map equals the orignal map.

\begin{proposition}
    \label{prop:phi-from-mono}
    Given a monotone endomap $f : L \to L$ there is a local generator $\phi : \mathcal{P}_\mathcal{U \sqcup V^+} (B \times L)$ with $\Gamma_\phi (x) = f(x)$, for all $x : L$.
\end{proposition}

\begin{proof}
    First we define $\phi : \mathcal{P}_\mathcal{U \sqcup V^+} (B \times L)$ via $\phi(b , a) : \equiv b \leq^B f(a)$ (technically we should lift to the universe $\mathcal{U \sqcup V^+}$). To show that $\phi$ is local it suffices to show the equivalence $\dBV f(x) \simeq S_{\phi , x}$ for any $x : L$. With the above equivalence established, we also may conclude
    $$\Gamma_\phi(x) \equiv \bigvee S_{\phi , x} = \bigvee \dBV f(x) = f(x)$$
    where the first equality is \Cref{def:Gamma}, the second follows from \Cref{prop:surj-same-sup} and the last from the small basis assumption (see \Cref{def:basis}). Now we need to actually provide the equivalence. For this it suffices to show for any $b : B$
    $$b \leq^B f(x) \longleftrightarrow \exists a' : L , b \leq^B f(a') \land a' \leq x$$
    which expresses that the two subsets have the same elements.
    For the forward direction we notice that $x$ does the trick as reflexivity gives $x \leq x$. For the other direction, suppose we have $a' : L$ with $b \leq^B f(a')$ and $a' \leq x$. By monotonicity, we have $f(a') \leq f(x)$. Then we can apply transitivity to conclude $b \leq^B f(x)$ as desired.
\end{proof}

\begin{remark}
    Notice that in the last step of the proof we actually first translate $b \leq^B f(a')$ to $\beta(b) \leq f(a')$ via the equivalence established in \Cref{def:basis}. Then apply transitivity and translate back along the equivalence. These details are important in formalization, but tedious to notate in proof sketches. 
\end{remark}

\begin{example}
    \label{ex:local-BSG's}
    A wealth of examples of local generators follow from Proposition 5.6. For example, consider the identity maps on either $\mathcal{V}$-generated sup lattice $\Omega_\mathcal{V}$ or $\mathcal{P}_\mathcal{V}(A)$ where $A : \mathcal{V}$ is a set.
\end{example}

\section{Least Fixed Point of Monotone Operators}
\label{sec:7}
In \cite{Curi} it is shown that there is a correspondence between $\text{c}_L$ and $\phi$-closed subsets and elements $a \in L$ such that $\Gamma_\Phi (a) \leq a$. Further it is shown that if $\mathcal{I}(\Phi)$ is a set then $\Gamma_\Phi$ has a least fixed point. We now translate these results into type theory.

Assuming a generator $\phi$ is local we can show that the $\mathcal{V}$-small subsets that are both c-closed and $\phi$-closed correspond to deflationary points with respect to $\Gamma_\phi$. For the remainder of this section we work in a context with a local generator $\phi$.

\begin{definition}
    \label{def:closure}
    Given a small subset $P : \mathcal{P}_\mathcal{V} (B)$ we say $P$ is \textbf{closed under containment and} $\phi$ if it satisfies analgous conditions to that of the least closed subset. Explicitly we write
    \begin{align*}
        \text{is-c-}\phi\text{-closed}_\mathcal{V} (P) &:\equiv \text{is-c-closed} (P) \times \text{is-}\phi\text{-closed} (P). 
    \end{align*}
\end{definition}

\begin{definition}
    \label{def:deflationary}
    A point $a : L$ is \textbf{deflationary} with respect to $\Gamma_\phi$ if $\Gamma_\phi (a) \leq a$. We will write $\text{is-deflationary}(a) : \equiv \Gamma_\phi (a) \leq a$
\end{definition}

One can observe that $\text{is-c-}\phi\text{-closed} (P)$ and $\text{is-deflationary} (a)$ are propositions as they are made up of propositions. We now prove the proposed correspondence, which reveals itself in type theory as an equivalence.

\begin{proposition}
    \label{prop:correspondance}
    The type of small c-$\phi$-closed subsets is equivalent to the type of deflationary points. That is,
    $$\sum_{P : \mathcal{P}_\mathcal{V}(B)} \text{is-c-$\phi$-closed} (P) \simeq \sum_{a : L} \text{is-deflationary} (a).$$
\end{proposition}

\begin{proof}
    We commence by defining maps in either direction. First suppose we have $P : \mathcal{P}_\mathcal{V} (B)$ that is c-$\phi$-closed. We will now show that $\bigvee P$ is deflationary: $\Gamma_\phi (\bigvee P) \leq \bigvee P$. For this, it suffices to show that $\bigvee P$ is an upper bound of $S_{\phi , \bigvee P}$. So, consider $b : B$ such that $\exists a' : L, (b , a') \in \phi \land a' \leq \bigvee P$ with the intent to show $q(b) \leq \bigvee P$. By definition of $\bigvee P$, this can be reduced to showing that $b \in P$. So, assume we have $a' : L$ with $(b , a') \in \phi$ and $a' \leq \bigvee P$. Now since $P$ is assumed to be c-$\phi$-closed we have 
    $$\text{is-c-closed}(P) \equiv \prod_{U : \mathcal{P}_{\mathcal{V}} (B)} U \subseteq P \to \dBV \bigvee U \subseteq P.$$ 
    If we use $P$ for $U$ and the fact that $a' \leq \bigvee P$ we conclude from c-closure that
    $$\dB a' \subseteq P.$$
    We also have
    $$\text{is-}\phi\text{-closed}(P) \equiv \prod_{a : L} \prod_{b : B} (b , a) \in \phi \to \dB a \subseteq P \to b \in P.$$ 
    We can now satisfy each hypothesis of the above, so we may conclude $b \in P$, as desired. 
    
    Now assume we have a deflationary point $a : L$. We now show that $\dBV a$ is c-$\phi$-closed. For c-closure we need a function $\Pi_{U : \mathcal{P}_\mathcal{V} (B)} U \subseteq \dBV a \to \dBV \bigvee U \subseteq \dBV a$. In the interest of defining such a function consider $U : \mathcal{P}_\mathcal{V}(B)$, with $U \subseteq \dBV a$ and $b : B$ with $b \in \dBV \bigvee U$. By assumption and \Cref{prop:subset-preserve-sup} 
    $$b \leq^B \bigvee U \leq \bigvee \dBV a = a$$
    so that $b \in \dBV a$, as desired. For $\phi$-closure we need a function of type $\Pi_{a' : L} \Pi_{b : B} (b , a') \in \phi \to \dBV a' \subseteq~ \dBV a \to b \leq^B a$. By \Cref{prop:subset-preserve-sup} we have $$a' = \bigvee \dBV a' \leq \bigvee \dBV a = a.$$ 
    From this we may conclude $\exists a', (b , a') \in \phi \land a' \leq a$, or $b \in S_{\phi , a}$. Now together with the assumption that $a$ is deflationary we conclude that
    $$b \leq^B \bigvee S_{\phi,a} \equiv \Gamma_\phi(a) \leq a$$
    as desired.
    
    It remains to show that these maps are inverse. We explicate the above maps $(P , \_) \mapsto (\bigvee P , \_)$ and $(a , \_) \mapsto (\dBV a , \_)$ where the second components are unspecified because they are propositionally valued. For this reason we need only check that the first projections are equal when the compositions are applied. That is, we must show 
    $$\bigvee \dBV a = a$$
    and 
    $$\dBV \bigvee P = P.$$
    Now, the former holds by \Cref{def:basis}. For the latter, it suffices to show $x \leq^B \bigvee P \longleftrightarrow x \in P$. The forward direction follows from c-closure. The reverse direction is immediate since $\bigvee P$ is an upper bound of all $x : L$ with $x \in P$. This concludes the proof.
\end{proof}

We can now show, under certain smallness assumption on the QIT family $\mathcal{I}_\phi$, that $\Gamma_\phi$ has a least fixed point. First we collect some important constructions that can be built under the assumptions of the following theorem.

If for all $b : B$ the type $b \in \mathcal{I}_\phi$ is $\mathcal{V}$-small, then we can construct the following:
\begin{enumerate}
    \item For each $b : B$, a type $\mathcal{I}_\phi^{\mathcal{V}} (b) : \mathcal{V}$ that is a proposition and an equivalence $\mathcal{I}_\phi^{\mathcal{V}} (b) \simeq b \in \mathcal{I}_\phi$.
    \item A subset $\mathcal{I}_\phi^{\mathcal{V}} : \mathcal{P}_\mathcal{V} (B)$ that witnesses to the fact that $\mathbb{T} (\mathcal{I}_\phi)$ is $\mathcal{V}$-small; that is, $\mathbb{T}(\mathcal{I}_\phi^{\mathcal{V}}) \simeq \mathbb{T} (\mathcal{I}_\phi)$.
    \item Intiality of $\mathcal{I}_\phi^{\mathcal{V}}$, which follows from that of $\mathcal{I}_\phi$ via the equivalence.
\end{enumerate}

\begin{proposition}
    \label{prop:small-QIT-give-fixed-point}
    Suppose that for all $b : B$ the type $b \in \mathcal{I}_\phi$ is $\mathcal{V}$-small. Then $\Gamma_\phi$ has a least fixed point. That is, there is a $p : L$ such that $\Gamma_\phi (p) = p$ and if $x : L$ with $\Gamma_\phi (x) = x$ then $p \leq x$.
\end{proposition}

\begin{proof}
    We now show that $p :\equiv \bigvee \mathcal{I}_\phi^{\mathcal{V}}$ is the least fixed point. To show it is a fixed point it suffices to show that $\Gamma_\phi (p) \leq p$ and $p \leq \Gamma_\phi (p)$. For the former we observe that since $\mathcal{I}_\phi^{\mathcal{V}}$ is c-$\phi$-closed we have that $p$ is deflationary, by \Cref{prop:correspondance}. For the latter, first observe that monotonicity, in  tandem with what we have just shown, establishes $\Gamma_\phi (\Gamma_\phi (p)) \leq 
    \Gamma_\phi (p)$. Thus, $\Gamma_\phi (p)$ is deflationary and by \Cref{prop:correspondance} yields a c-$\phi$-closed subset $\dBV \Gamma_\phi (p)$. By the inherited initiality, $\mathcal{I}_\phi^{\mathcal{V}} \subseteq \dBV \Gamma_\phi (p)$. Finally, by \Cref{prop:subset-preserve-sup} 
    $$p \equiv \bigvee \mathcal{I}_\phi^{\mathcal{V}} \leq \bigvee \dBV \Gamma_\phi (p) = \Gamma_\phi (p),$$
    where the final equality follows from \Cref{prop:correspondance}, where we showed the maps $\bigvee \_$ and $\dBV \_$ are inverses. Now to complete the proof, let $x : L$ be any other fixed point. In particular, $x$ is deflationary and as such yields a c-$\phi$-closed subset $\dBV x$. As $\mathcal{I}_\phi^{\mathcal{V}}$ is initial we have $\mathcal{I}_\phi^{\mathcal{V}} \subseteq \dBV x$ and thus by \Cref{prop:subset-preserve-sup} 
    $$p \equiv \bigvee \mathcal{I}_\phi^{\mathcal{V}} \leq \bigvee \dBV x = x$$
    where once again the last equality follows from \Cref{prop:correspondance}.
\end{proof}

\begin{remark}
    It is worth mentioning that the smallness assumption we made in \Cref{prop:small-QIT-give-fixed-point} is to ensure that $\bigvee \mathbb{T} (\mathcal{I}_\phi^{\mathcal{V}})$ exists. In fact, we may now record the following corollary which holds under the assumption of propositional resizing.
\end{remark}

In light of the previous result we gain an impredicative version of the least fixed point theorem which follows from propositional resizing.

\begin{corollary}
    \label{cor:Impredicative-Least-Fixed-Point}
    If $\text{Prop-Resizing}_\mathcal{U \sqcup V^+,V}$ holds then every monotone endomap $f : L \to L$ has a least fixed point.
\end{corollary}

\begin{proof}
    By \Cref{prop:phi-from-mono} there is a generator $\phi$ and monotone operator $\Gamma_\phi$ that corresponds to $f$. By $\text{Prop-Resizing}_\mathcal{U \sqcup V^+,V}$ we can show that $b \in \mathcal{I}_\phi$ is $\mathcal{V}$-small for every $b : B$. Thus, \Cref{prop:small-QIT-give-fixed-point} gives a least fixed point of $\Gamma_\phi$ and, by extension, $f$.
\end{proof}

\section{Bounded Inductive Generators}
\label{sec:8}
In \cite{Curi} a further restriction on abstract inductive definitions is imposed. Given an abstract inductive definition $\Phi$ we say it is bounded if
\begin{enumerate}
    \item $\{b \in B \ | \ (b , a) \in \Phi \}$ is a set for every $a \in L$
    \item There is a set $\alpha$ such that, whenever $(b , a) \in \Phi$ there is $x \in \alpha$ such that the set $\downarrow^B a$ is the image of $x$.
\end{enumerate}
Notice if $\Phi$ is a set, rather than a class, then $\Phi$ is automatically bounded.

We will now explore the translation of these notions into type theory. For the remainder of this section we work in the context of a $\mathcal{V}$-generated sup lattice $L$ with basis $\beta : B \to L$.
 
\begin{definition}
    \label{def:small-covering}
    A type $X : \mathcal{V}$ is a \textbf{small covering} of $Y: \mathcal{T}$ if there is a surjection $s : X \twoheadrightarrow Y$.
\end{definition}

\begin{definition}
    \label{def:has-bound}
    We say that a generator $\phi$ has a \textbf{bound} if there is $I : \mathcal{V}$ and $\alpha : I \to \mathcal{V}$ such that for any $a : L$ and $b : B$ with $(b , a) \in \phi$ there merely exists $i : I$ such that $\alpha(i)$ is a small covering of $\downarrow^B a$, that is
    $$\exists i : I , \alpha (i) \twoheadrightarrow \dB a.$$
\end{definition}

\begin{definition}
    \label{def:bounded}
    We say that a generator $\phi$ is \textbf{bounded} if for any $a : L$ and $b : B$ the type $(b , a) \in \phi$ is $\mathcal{V}$-small and $\phi$ has a bound.
\end{definition}

\begin{definition}
    \label{def:small-phi}
    We say that a generator $\phi$ is \textbf{small} if for any $a : L$ and $b : B$ the type $(b , a) \in \phi$ is $\mathcal{V}$-small and $\mathbb{T} (\phi) \equiv \sum_{(b,a) : B \times L} (b,a) \in \phi$ is $\mathcal{V}$-small.
\end{definition}

\begin{remark}
    If $\phi$ is small then it is bounded. We have as a bound $ \alpha' : \mathbb{T} (\phi) \to \mathcal{V}$ defined via
    $$(b , a , \_) \mapsto \dBV a.$$
    Technically, we should give $\alpha : T \to \mathcal{V}$ where $T : \mathcal{V}$ and $T \simeq \mathbb{T}(\phi)$.
\end{remark}

We now show that if a generator is bounded then it is local (see \Cref{def:local}).

\begin{proposition}
    \label{prop:bounded-implies-local}
    Every bounded generator is local.
\end{proposition}

\begin{proof}
    Let $\phi$ be a bounded generator with bound $\alpha$. To show $\phi$ is local it suffices to show 
    $$\left(\sum_{b : B} \exists i : I , \exists m : \alpha(i) \to \dB a , (b , \bigvee \alpha(i)) \in \phi\right) \simeq S_{\phi , a}$$
    for any $a : L$; since the type on the left hand side is itself $\mathcal{V}$-small under the assumptions. For this, it suffices to show that 
    $$\exists i : I , \exists m : \alpha(i) \to \downarrow^B a , (b , \bigvee \alpha(i)) \in \phi \longleftrightarrow \exists a' : L , (b , a') \in \phi \land a' \leq a$$
    for any $b : B$.
    
    For the forward direction, assume we have $i : I$ and $m : \alpha(i) \to \downarrow^B a$ such that $(b , \bigvee \alpha(i)) \in \phi$. Notice that for any $z : \alpha(i)$ we have that $\beta(\text{pr}_1(m(z))) \leq a$. This shows that $a$ is an upper bound of $\alpha(i)$ and as such $\bigvee \alpha(i) \leq a$, as desired. So we set $a' :\equiv \bigvee \alpha(i)$.
    
    For the reverse direction, assume $a' : L$ such that $(b , a') \in \phi$ and $a' \leq a$. First observe that $a' \leq a$ yields an inclusion $\iota : \dB a' \to \dB a$. Since $\phi$ is bounded, $\exists i : I, \alpha(i) \twoheadrightarrow \dB a'$. Let $f : \alpha(i) \to \dB a'$ be the function underlying the surjection and put $m : \equiv \iota \circ f : \alpha(i) \to \dB a$. Since $\iota$ is merely an inclusion we determine $a' = \bigvee \alpha(i)$ by \Cref{prop:surj-same-sup}. Finally, since $(b , a') \in \phi$, we conclude $(b, \bigvee \alpha(i)) \in \phi$.
\end{proof}

Recall, \Cref{prop:phi-from-mono} states that every monotone map determines canonical a local generator: $(b , a) \in \phi : \equiv b \leq^B f(a)$. Unfortunately, this generator is not particularly well behaved. To illustrate this we give the following example. 

\begin{example}
    \label{ex:cannon-gen-poor}
    Consider the map $f : L \to L$ defined to be constant at some $c : L$. The canonical local generator determined by $f$ is given by $(b , a) \in \phi :\equiv b \leq^B c$. To show this generator is bounded we need to give a small family $\alpha : I \to \mathcal{V}$ such that for any $b : B$ and $a : L$, if $b \leq^B c$ then $\exists i : I, \alpha(i) \twoheadrightarrow \dB a$. Assuming the hypothesis is satisfied by any $b : B$, it is satisfied by every $a : L$. It would very difficult (likely impossible) to provide a small family that covers $\dB a$ for every $a$. 

    Alternatively, we could define a generator $\phi$ which determines the function $f$. First, let's enforce the reasonable assumption that $L$ is locally $\mathcal{V}$-small. Let $(b , a) \in \phi : \equiv b \leq^B c \land a = 0_L$, where $0_L :\equiv \bigvee \emptyset$. Notice,
    \begin{align*}
        S_{\phi,a} &\equiv \sum_{b : B} \exists a' : L, b \leq^B c \land a' = 0_L \land a' \leq a \\
        &\simeq \sum_{b : B} b \leq^B c \land 0_L \leq a \\
        &\simeq \sum_{b : B} b \leq^B c \\
        &\equiv \downarrow^B_\mathcal{V} c
    \end{align*}
    so $\Gamma_\phi(a) = c$. Moreover, this generator is bounded since $b \leq^B c \land a = 0_L$ is $\mathcal{V}$-small for all $a : L$ and $b : B$ and 
    \begin{align*}
        \mathbb{T}(\phi) &\simeq \sum_{b : B} \sum_{a : L} b \leq^B c \land a = 0_L \\
        &\simeq \sum_{b : B} b \leq^B c \times \sum_{a : L} a = 0_L \\
        &\simeq \sum_{b : B} b \leq^B c \ \ , \ \text{since} \ \sum_{a : L} a = 0_L \ \text{is contractible}
    \end{align*}
    is $\mathcal{V}$-small. Thus, $\phi$ is bounded. 

\end{example}

A similar situations occurs for other simple functions like the identity function. In light of this peculiarity we may extend the boundedness restriction to monotone endomaps as follows.

\begin{definition}
    \label{def:f-bounded}
    We say a monotone endomap $f : L \to L$ is \textbf{bounded} if there exists a bounded generator $\phi : \mathcal{P}_\mathcal{U \sqcup V^+}(B \times L)$ with $\Gamma_\phi(x) = f(x)$, for all $x : L$.
\end{definition}

\section{Small-Presentation of a Lattice}
\label{sec:9}
In \cite{Curi} a further restriction of set-generated complete lattices is explored. We depart from Curi here and instead use an equivalent formulation of set presentation due to \cite{Aczel}. A complete lattice $L$ is set-presented if there is a subset $R \subseteq B \times \mathcal{P}(B)$ such that for any $b \in B$ and $X \subseteq B$
$$b \leq \bigvee X$$
iff 
$$\exists Y \subseteq X , (b , Y) \in R $$

We now explore the translation of this notion into type theory. For the remainder of this section we work in the context of a $\mathcal{V}$-generated sup lattice $L$ with basis $\beta : B \to L$.

\begin{definition}
    \label{def:small-presentation}
    A subset $R : \mathcal{P}_\mathcal{V} (B \times \mathcal{P}_\mathcal{V} (B))$ together with a family $Y : J \to \mathcal{P}_\mathcal{V}(B)$ for $J : \mathcal{V}$ is a $\mathcal{V}$\textbf{-presentation} of $L$ if for any $b : B$ and $X : \mathcal{P}_\mathcal{V} (B)$ 
    $$b \leq^B \bigvee X \longleftrightarrow \exists j : J , Y(j) \subseteq X \land (b , Y(j)) \in R.$$
    Finally, we say a $\mathcal{V}$-generated sup lattice with a basis $\beta : B \to L$ is $\mathcal{V}$\textbf{-presented} if it has a $\mathcal{V}$-presentation. For notational convenience, we also define a subset $R_j : \mathcal{P}_\mathcal{V}(B)$ by $b \mapsto (b , Y_j) \in R$ for each $j : J$.
\end{definition}

\begin{example}
    \label{ex:subset-small-pres}
    Recall, $\mathcal{P}_\mathcal{V}(A)$, where $A : \mathcal{V}$ is a set, has a basis given by $A$ and the map $\{ \_ \} : A \to \mathcal{P}_\mathcal{V} (A)$. Now $\_ \in \_ : \mathcal{P}_\mathcal{V}(A \times \mathcal{P}_\mathcal{V} (A))$ together with the same map $\{\_\} : A \to \mathcal{P}_\mathcal{V}(A)$ provides a small-presentation. To see this let $a : A$ and $X : \mathcal{P}_\mathcal{V}(A)$. Since $\bigvee X \equiv \bigcup \{\_\} \circ \text{inc}_X = X$, it suffices to show 
    $$\{a\} \subseteq X \longleftrightarrow \exists x : A , (\{x\} \subseteq X \land a \in \{x\}).$$
    Which is trivially true by taking $x :\equiv a$. 
\end{example}

\section{Predicative Least Fixed Point Theorem}
\label{sec:10}
In \cite{Curi} it is shown that given a set-presented complete lattice $L$ and a monotone endomap $f : L \to L$, if there exists a bounded abstract inductive definition $\Phi$ with $f (x) = \Gamma_\Phi (x)$, then $f$ has a least fixed point. This result hinges on showing that under these assumptions the least closed class $\mathcal{I}(\Phi)$ is in fact a set, then the least fixed point is $\bigvee \mathcal{I}(\Phi)$. We will now translate these results into type theory.

We will follow \cite{Curi} and show that the assumptions of $\mathcal{V}$-presentation and boundedness provide us with smallness of $\mathcal{I}_\phi$. In satisfying this smallness assumption there are two obstacles. The first obstacle is that the constructor c-cl of $\mathcal{I}_\phi$ quantifies over $\mathcal{P}_\mathcal{V}(B)$ which lives in the universe $\mathcal{V^+}$. This large quantification can be avoided by appealing to the $\mathcal{V}$-presentation assumption. The second obstacle is that the constructor $\phi$-cl of $\mathcal{I}_\phi$ quantifies over the carrier $L$ which lives in the universe $\mathcal{U}$. This large quantification can be avoided by appealing to the $\phi$-bounded assumption. In fact, we will now define a new QIT family and show that it is equivalent to $\mathcal{I}_\phi$ --- with in the context of the above assumptions. 

\begin{definition}
    \label{def:Small-QIT}
    Let $L$ be a $\mathcal{V}$-sup-lattice with a $\mathcal{V}$-presentation $(J , Y , R)$ and a bounded generator $\phi$ with bound $(I , \alpha)$. For any $b : B$ and $a : L$, let $(b , a)^\mathcal{V} \in \phi$ be the $\mathcal{V}$-small type guaranteed by the assumption that $\phi$ is bounded. First, we define similar shorthand as before. For $S : \mathcal{P}_\mathcal{T}(B)$
    $$\text{is-c-closed}^\mathcal{V}(S) :\equiv \prod_{j : J} Y(j) \subseteq S \to R_j \subseteq S $$
    and 
    $$\text{is-}\phi\text{-closed}^\mathcal{V}(S) \equiv \prod_{b : B} \prod_{i : I} \prod_{m : \alpha(i) \to B} (b , \bigvee \alpha(i))^\mathcal{V} \in \phi 
    \to \dBV \bigvee \alpha(i) \subseteq S 
    \to b \in S.$$
    We define the QIT family $\mathcal{I}^\mathcal{V}_\phi : B \to \mathcal{V}$, called the \textbf{least closed small subset under containment and} $\phi$, which has the following constructors:
    \begin{enumerate}
        \item $\mathcal{I}^\mathcal{V}\text{-trunc} : \prod_{b : B} \text{is-prop} (\mathcal{I}^\mathcal{V}_\phi (b))$,
        \item $\text{c-cl}^\mathcal{V} : \text{is-c-closed}^\mathcal{V}(\mathcal{I}^\mathcal{V}_\phi)$,
        \item $\phi\text{-cl}^\mathcal{V} : \text{is-}\phi\text{-closed}^\mathcal{V}(\mathcal{I}^\mathcal{V}_\phi)$.
    \end{enumerate}
    The $\mathcal{I}^\mathcal{V}\text{-trunc}$ constructor guarantees that $\mathcal{I}^\mathcal{V}_\phi : \mathcal{P}_\mathcal{V}(B)$. We will state the initiality principle of $\mathcal{I}^\mathcal{V}_\phi$, just as we did in \Cref{def:QIT}:
    $$\mathcal{I}^\mathcal{V}\text{-initial} : \prod_{P : \mathcal{P}_\mathcal{T}(B)} \text{is-c-closed}^\mathcal{V} (P) \to \text{is-}\phi\text{-closed}^\mathcal{V}(P) \to \mathcal{I}^\mathcal{V}_\phi \subseteq P.$$
\end{definition}

We will now show that, under the assumptions of small presentation and boundedness, the total spaces of $\mathcal{I}^\mathcal{V}_\phi$ and $\mathcal{I}_\phi$ are equivalent. This is because, under these assumptions, the constructors of either QIT family are inter-derivable.

\begin{proposition}
    \label{prop:QIT-is-small}
    For a $\mathcal{V}$-sup-lattice $L$ with a $\mathcal{V}$-presentation and a bounded generator $\phi$, the type $b \in \mathcal{I}_\phi$ is $\mathcal{V}$-small for all $b : B$. 
\end{proposition}

\begin{proof}
    Suppose that $(J , Y , R)$ is a $\mathcal{V}$-presentation of $L$ and $(I , \alpha)$ is a bound for $\phi$. It suffices to show $b \in \mathcal{I}_\phi \longleftrightarrow b \in \mathcal{I}^\mathcal{V}_\phi$, for all $b : B$. For this we can simply show $\mathcal{I}_\phi \subseteq \mathcal{I}^\mathcal{V}_\phi$ and $\mathcal{I}^\mathcal{V}_\phi \subseteq \mathcal{I}_\phi$, utilizing the initiality of either subset.
    
    For the former, we will apply $\mathcal{I}$-initial to $\mathcal{I}^\mathcal{V}_\phi$ as well as two functions
    $$\prod_{U : \mathcal{P}_\mathcal{V}(B)} U \subseteq \mathcal{I}^\mathcal{V}_\phi \to \dBV \bigvee U \subseteq \mathcal{I}^\mathcal{V}_\phi$$
    and
    $$\prod_{a : L} \prod_{b : B} (b , a) \in \phi \to \dBV a \subseteq \mathcal{I}^\mathcal{V}_\phi \to b \in \mathcal{I}^\mathcal{V}_\phi.$$
    In the interest of defining the first function, we assume $U : \mathcal{P}_\mathcal{V}(B)$ with $U \subseteq \mathcal{I}^\mathcal{V}_\phi$ as well as $b : B$ with $b \leq^B \bigvee U$. By $\mathcal{V}$-presentation there exists $j : J$ such that $Y(j) \subseteq U$ and $(b , Y(j)) \in R$ (or equivalently $b \in R_j$). Now, by transitivity, we have $Y(j) \subseteq \mathcal{I}^\mathcal{V}_\phi$, so we can apply $\text{c-cl}^\mathcal{V}$ to conclude $b \in \mathcal{I}^\mathcal{V}_\phi$, as desired. In the interest of defining the second function, we assume $a : L$ and $b : B$ with $(b , a) \in \phi$ and $\dBV a \subseteq \mathcal{I}^\mathcal{V}_\phi$. By $\phi$-boundedness there exists $i : I$ such that $s : \alpha(i) \twoheadrightarrow \dB a$. By \Cref{def:basis} and \Cref{prop:surj-same-sup} we have $$a = \bigvee \dB a = \bigvee \alpha(i).$$
    Thus, we conclude $(b , \bigvee \alpha(i)) \in \phi$ and $\dBV \bigvee \alpha(i) \subseteq \mathcal{I}^\mathcal{V}_\phi$. Now, by the equivalence we have $(b , \bigvee \alpha(i))^{\mathcal{V}} \in \phi$ so we can apply $\phi$-cl$^\mathcal{V}$ to conclude $b \in \mathcal{I}^\mathcal{V}_\phi$, as desired.

    For the latter, we will apply $\mathcal{I}^\mathcal{V}$-initial to $\mathcal{I}_\phi$ as well as two functions
    $$\prod_{j : J} Y(j) \subseteq \mathcal{I}_\phi \to R_j \subseteq \mathcal{I}_\phi$$
    and
    $$\prod_{b : B} \prod_{i : I} \prod_{m : \alpha(i) \to B} (b , \bigvee \alpha(i))^\mathcal{V} \in \phi \to \dBV \bigvee \alpha(i) \subseteq \mathcal{I}_\phi \to b \in \mathcal{I}_\phi.$$
    In the interest of defining the first function, we assume $j : J$ with $Y(j) \subseteq \mathcal{I}_\phi$ as well as $b : B$ with $(b , Y(j)) \in R$. Of course, $Y(j) \subseteq Y(j)$ so by $\mathcal{V}$-presentation we have $b \leq^B \bigvee Y(j)$. Thus, by c-cl we conclude $b \in \mathcal{I}_\phi$, as desired. In the interest of defining the second function, we assume $b : B$, $i : I$ and $m : \alpha(i) \to B$ with $(b , \bigvee \alpha(i))^\mathcal{V} \in \phi$ and $\dBV \bigvee \alpha(i) \subseteq \mathcal{I}_\phi$. Of course, $\bigvee \alpha(i) : L$ and $(b , \bigvee \alpha(i)) \in \phi$ so we immediately conclude by $\phi$-cl that $b \in \mathcal{I}_\phi$, as desired. This completes the proof.
\end{proof}

As a corollary to the above proposition we get a predicative version of Tarski's Least Fixed Point Theorem.

\begin{corollary}
    \label{cor:Tarski}
    Let $L$ be a $\mathcal{V}$-presented sup lattice and $f : L \to L$ a monotone endomap. If $f$ is bounded (see \Cref{def:f-bounded}), then $f$ has a least fixed point.
\end{corollary} 

\begin{proof}
    By \Cref{prop:QIT-is-small} we have that $b \in \mathcal{I}_\phi$ is $\mathcal{V}$-small for all $b : B$. Now applying \Cref{prop:small-QIT-give-fixed-point} we conclude that $\Gamma_\phi$, and by extension $f$, has a least fixed point.
\end{proof}

Inspired by \Cref{ex:cannon-gen-poor} we will now investigate a condition that guarantees a monotone map is bounded with respect to a generator. This will allow us to give another version of the least fixed point theorem.

\begin{definition}
    \label{def:dense}
    Let $L$ be a $\mathcal{V}$-generated sup lattice. We say a monotone map $f : L \to L$ is \textbf{dense} if there is a family $\gamma : V \to L$ with $V : \mathcal{V}$ such that for all $b : B$ and $a : L$
    $$b \leq^B f(a) \to \exists v : V, b \leq^B f(\gamma(v)) \land \gamma(v) \leq a.$$
    We call $\gamma$ the \textbf{dense family}.
\end{definition}

\begin{remark}
    \label{rem:5}
    The reader may be asking themselves why we introduced the dense family $\gamma : V \to L$ when we already have a basis family. Recall, that many concrete examples of $\mathcal{V}$-generated sup lattice have multiple different possible bases --- some of which are more natural than others. For example, the canonical basis for the powerset is the singleton basis. Notice for a monotone map $f : \mathcal{P}_\mathcal{V}(A) \to \mathcal{P}_\mathcal{V}(A)$, where $A$ is a set, if we use the singleton basis as the dense family then $f$ is dense if for any $a : A$ and $S : \mathcal{P}_\mathcal{V}(A)$
    $$a \in f(S) \to \exists x : A, a \in f(\{x\}) \land x \in S.$$ Notice this implies that $f(\emptyset) = \emptyset$, since $a \in f(\emptyset) \to \bot$. Hence if $f$ is dense with respect to the singleton basis then $f$ has a trivial least fixed point and as such is not very insightful.

    Alternatively, there is another basis $\beta : \text{List} (A) \to \mathcal{P}_\mathcal{V} (A)$, the finite subset basis (see \Cref{ex:pow-basis}). Notice when we allow $\beta$ to be the dense family, then $f$ is dense if for any $b : B$ and $S : \mathcal{P}_\mathcal{V} (A)$
    $$b \in f(S) \to \exists l : \text{List}(A), b \in f(\beta(l)) \land \beta(l) \subseteq S.$$
    Interestingly, density with respect to $\text{List} (A)$ is equivalent to Scott continuity, which says $f(\bigcup \alpha) = \bigcup f \circ \alpha$, where $\alpha$ is a directed family (see \cite[Section 3.3]{Jong2023DomainTI}). 

    Finally, note that there may be other families, unrelated to the basis, that provide even more interesting density conditions. For these reasons we leave the choice of basis and dense family independent.
\end{remark}

In light of \Cref{rem:5} the reader may be curious about the strength of density as an assumption. The next example will showcase that density, in the case of the sup lattice $\mathcal{P}_\mathcal{V}(A)$, where $A : \mathcal{V}$ is a set, is a strictly weaker assumption than Scott continuity.

\begin{example}
    \label{ex:Dense-vs-Scott}
    For the purpose of this example we say a set is countable if it is equipped with an enumeration. Analogous to the equivalence of Scott continuity and density with respect to $\text{List}(A)$ (see \Cref{rem:5}) we can show that density with respect to countable subsets of $A$ is equivalent to continuity of countably directed families. The precise definition of the dense family $\gamma : (\mathbb{N} \to A + 1) \to \mathcal{P}_\mathcal{V}(A)$ is given by
    $$z \mapsto \bigcup \bar{\{\_\}} \circ z$$ 
    where $\bar{\{\_\}} : A + 1 \to \mathcal{P}_\mathcal{V}(A)$ via $a \mapsto \{a\}$ or $\star \mapsto \emptyset$.
    
    Let $f : \mathcal{P}_\mathcal{V}(A) \to \mathcal{P}_\mathcal{V}(A)$ be dense with respect to $\gamma$ and $\alpha : I \to \mathcal{P}_\mathcal{V}(A)$ be a countably directed family. Using monotonicity it is routine to argue that $\bigcup f \circ \alpha \subseteq f(\bigcup \alpha)$. To show the other containment we suppose $x \in f(\bigcup \alpha)$ and notice by density
    $$\exists z : \mathbb{N} \to A + 1, x \in f(\gamma(z)) \land \gamma(z) \subseteq \bigcup \alpha.$$
    Since $\gamma(z)$ is countable we can apply the countable directedness of $\alpha$ to conclude $\exists i : I, \gamma(z) \subseteq \alpha(i)$, which in turn would yield
    $$\exists i : I, x \in f(\alpha(i))$$
    so we conclude that $x \in \bigcup f \circ \alpha$. Conversely, suppose $f$ preserves countably directed unions and let $x : A$ and $S : \mathcal{P}_\mathcal{V}(A)$ with $x \in f(S)$. Now, we can define $s : (\mathbb{N} \to \mathbb{T}(S) + 1) \to \mathcal{P}_\mathcal{V}(A)$ as follows
    $$z \mapsto \bigcup \bar{\{\_\}} \circ z$$
    which is the family of countable subsets of $S$. This family $s$ is countably directed (here we make essential use of the fact that countable sets come equipped with enumerations to show that the countable union of countable sets is countable) and $S = \bigcup s$. By assumption of preservation of countably directed unions we have $x \in \bigcup f \circ s$. Thus, there is $z' : \mathbb{N} \to \mathbb{T}(S) + 1$ such that $x \in f(s(z'))$ and $s(z') \subseteq S$. We can convert $z'$ to a $z : \mathbb{N} \to A + 1$ such that $s(z') = \gamma(z)$. Hence, 
    $$\exists z : \mathbb{N} \to A + 1, x \in f(\gamma(z)) \land \gamma(z) \subseteq S$$
    $f$ is dense with respect to $\gamma$.

\end{example}

In light of \Cref{ex:Dense-vs-Scott}, since preservation of countably directed families is strictly weaker than Scott continuity (?), we see that density is strictly weaker than Scott continuity. We now state our final result which shows that any dense monotone map is bounded and as a result has a least fixed point. 

\begin{proposition}
    \label{prop:dense-gives-bounded}
    Let $L$ be a $\mathcal{V}$-presented and locally $\mathcal{V}$-small sup lattice and $f : L \to L$ a monotone map. If $f$ is dense, then $f$ is bounded.
\end{proposition}

\begin{proof}
    We start by constructing a bounded generator and then we show it induces $f$. Define $\phi : \mathcal{P}_\mathcal{U \sqcup V^+}(B \times L)$ by $(b , a) \in \phi :\equiv \exists v : V, b \leq^B f(a) \land a = \gamma(v)$. Clearly, $(b , a) \in \phi$ is $\mathcal{V}$-small since $L$ is locally $\mathcal{V}$-small. Notice that $\alpha : V \to \mathcal{V}$ defined by $\alpha(x) :\equiv \dBV \gamma(x)$ is a bound. To see this, consider any $b : B$ and $a : L$. If $\exists v : V, b \leq^B f(a) \land a = \gamma(v)$ then $\exists v : V, \dBV \gamma(v) \simeq \dB a$ and this can be demoted to a surjection $\dBV \gamma(v) \twoheadrightarrow \dB a$. Finally, notice that 
    \begin{align*}
        S_{\phi,a} &\equiv \sum_{b : B} \exists a' : L, \exists v : V, b \leq^B f(a') \land a' = \gamma(v) \land a' \leq a \\
        &\simeq \sum_{b : B} \exists v : V, b \leq^B f(\gamma(v)) \land \gamma(v) \leq a \\
        &\simeq \sum_{b : B} b \leq^B f(a) \\
        &\equiv \downarrow^B_\mathcal{V} f(a)
    \end{align*}
    where the last equivalence only holds since $f$ is dense. Thus,
    $$\Gamma_\phi (a) = \bigvee S_{\phi,a} = \bigvee \dBV f(a) = f(a)$$
    which shows that $f$ is bounded. 
\end{proof}

\begin{corollary}
    \label{cor:Tarski-from-dense}
    Let $L$ be a $\mathcal{V}$-presented and locally $\mathcal{V}$-small sup lattice and $f : L \to L$ a monotone map. If $f$ is dense, then $f$ has a least fixed point.
\end{corollary}

\begin{proof}
    By \Cref{prop:dense-gives-bounded} we have that $f$ is bounded, so by \Cref{cor:Tarski} we conclude that $f$ has a least fixed point.
\end{proof}

\section{Conclusion}
\label{sec:11}
We have successfully translated a constructive and predicative version of the least fixed point theorem into type theory. It is apparent that inductive types, particularly quotient inductive types (QITs), played a crucial role in our development. Essentially, we are skirting the constructions of subsets that are initial with respect to certain closure properties. It is worth noting that QITs, and more generally higher inductive types (HITs), are recent developments and an area of active research. Nonetheless, they are often viewed as fitting with in the philosophical framework of (generalised) predicativity as a reasonable extension of inductive constructions. For this reason one may argue that QITs are constructively and predicatively admissible and as such this work provides a step towards the \say{system independent} derivation that Giovanni Curi called for in the conclusion of \cite{Curi}. 

There are many avenues of future research. One obvious question is can we completely salvage a predicative version of Tarski's fixed point theorem? Tarski's fixed point theorem states that the collection of fixed points of a monotone map forms a complete lattice. To achieve such a result we need to also construct a greatest fixed point. In fact, Curi has already done this in \cite{Curi2}. The approach dualizes the present work by considering coinductive definitions. We intend to explore how this idea translates into type theory via coinductive types. Following this, we would like to explore if the resulting complete lattice is itself small generated and small presented.

Avenues of further inquiry more directly related to the present work would be to determine how common $\mathcal{V}$-presented sup lattices and bounded maps really are. In light of \Cref{ex:cannon-gen-poor} it seems difficult to come up with interesting examples that satisfy the hypothesis of \Cref{cor:Tarski}. In the interest of showing the non-triviality of the result, we intend to explore and find some interesting examples and applications. Inspired by \Cref{prop:dense-gives-bounded}, we intend to investigate a streamlined approach to cooking up (bounded) generators that yield functions of interest. Lastly, we would like to apply the results of this note to other complete poset structures such frames, DCPOs, etc. If variations of this result hold for such structures we may find applications in the areas of Formal Topology and Domain Theory. There are two particularly promising directions to explore. The first is to give a predicative version of a result from locale theory: the nucleus of a frame has a least fixed point \cite[Section II.2]{Johnstone}. The other would be to investigate a predicative version of Pataria's fixed point theorem \cite{Esc}, which is essentially Tarski's fixed point theorem for DCPOs.

\pagebreak

\printbibliography[heading=bibintoc]

\end{document}